\documentclass[a4paper]{amsart}

\usepackage[T1]{fontenc}
\usepackage{enumerate, amsmath, amsfonts, amssymb, amsthm, mathrsfs, wasysym, graphics, graphicx, xcolor, url, hyperref, hypcap,  shuffle, xargs, multicol, overpic, pdflscape, multirow, hvfloat, minibox, accents, array, xifthen, a4wide, ae, aecompl, paralist, lmodern}
\hypersetup{colorlinks=true, citecolor=darkblue, linkcolor=darkblue}
\usepackage[all]{xy}
\usepackage[bottom]{footmisc}
\graphicspath{{figures/}}
\usepackage{caption}
\title{The serpent nest conjecture for accordion complexes}

\thanks{Supported by a French doctoral grant Gaspard Monge of the {\'E}cole polytechnique (Palaiseau, France) and partially supported by the French ANR grant SC3A (15 CE40 0004 01).}

\author{Thibault Manneville}
\address[Thibault Manneville]{LIX, \'Ecole Polytechnique}
\email{thibault.manneville@polytechnique.edu}
\urladdr{\url{http://www.lix.polytechnique.fr/~manneville/}}

\newtheorem{theorem}{Theorem}

\newtheorem{lemma}[theorem]{Lemma}
\newtheorem{conjecture}[theorem]{Conjecture}

\theoremstyle{definition}

\newtheorem{example}[theorem]{Example}
\newtheorem{remark}[theorem]{Remark}

\newcommand{\ssm}{\smallsetminus} 
\newcommand{\eqdef}{\mbox{\,\raisebox{0.2ex}{\scriptsize\ensuremath{\mathrm:}}\ensuremath{=}\,}} 

\newcommand{\fref}[1]{Figure~\ref{#1}} 
\definecolor{darkblue}{rgb}{0,0,0.7} 
\newcommand{\darkblue}{\color{darkblue}} 
\newcommand{\defn}[1]{\textsl{\darkblue #1}} 

\usepackage{todonotes}

\newcommand{\accordionComplex}{\mathcal{AC}} 
\newcommand{\polygon}{\mathcal{P}} 
\newcommand{\triangulation}{\mathrm{T}} 
\newcommand{\quadrangulation}{\mathrm{Q}} 
\newcommand{\dissection}{\mathrm{D}} 
\newcommand{\cell}{\mathrm{C}} 
\newcommand{\accordion}{\mathrm{A}} 
\newcommand{\zigzag}{\mathrm{Z}} 
\newcommand{\ex}{\mathrm{ex}} 
\renewcommand{\restriction}[2]{\left.\kern-\nulldelimiterspace #1 \vphantom{\big|} \right|_{#2}}
\newcommand{\dual}{\star}

\newcommandx{\setOfDiagonals}[1][1=S]{\mathrm{#1}}

\newcommand{\serpent}{\mathrm{S}}

\newcommand{\serpentNest}{\mathrm{N}}
\newcommand{\allSerpentNests}{\mathcal{SN}}
\newcommand{\bijectionDissectionsToSerpentNests}[1]{\Phi_{#1}}
\newcommand{\bijectionSerpentNestsToDissections}[1]{\Psi_{#1}}

\begin{document}
\phantom{a}
\vspace{-1.5cm}

\maketitle

\phantom{a}
\vspace{-.9cm}

\begin{abstract}
Consider $2n$ points on the unit circle and a reference dissection~$\dissection_\circ$ of the convex hull of the odd points. The accordion complex of~$\dissection_\circ$ is the simplicial complex of subsets of pairwise noncrossing diagonals with even endpoints that cross a connected set of diagonals of the dissection~$\dissection_\circ$. In particular, this complex is an associahedron when~$\dissection_\circ$ is a triangulation, and a Stokes complex when~$\dissection_\circ$ is a quadrangulation. We exhibit a bijection between the facets of the accordion complex of~$\dissection_\circ$ and some dual objects called the serpent nests of~$\dissection_\circ$. This confirms in particular a prediction of F.~Chapoton (2016) in the case of Stokes complexes.

\medskip
\noindent
\textsc{keywords.} Accordion complex $\cdot$ Dissections $\cdot$ Stokes complexes $\cdot$ Associahedron $\cdot$ Serpent nests.
\end{abstract}

\section{Introduction}
\label{sec:introduction}

\subsection{Motivations}
\label{subsec:motivations}

Y.~Baryshnikov introduced in~\cite{Baryshnikov} the definition of a~\defn{Stokes complex}, namely the simplicial complex of dissections of a polygon that are in some sense compatible with a reference quadrangulation~$\quadrangulation_\circ$. Although the precise definition of compatibility is a bit technical in~\cite{Baryshnikov}, it turns out that a diagonal is compatible with~$\quadrangulation_\circ$ if and only if it crosses a connected subset of diagonals of a slightly rotated version of~$\quadrangulation_\circ$, that we call an~\defn{accordion} of~$\quadrangulation_\circ$. We therefore also call Y.~Baryshnikov's simplicial complex the \defn{accordion complex}~$\accordionComplex(\quadrangulation_\circ)$ of~$\quadrangulation_\circ$. As an example, this complex coincides with the classical associahedron when all the diagonals of the reference quadrangulation~$\quadrangulation_\circ$ have a common endpoint. Revisiting some combinatorial and algebraic properties of~$\accordionComplex(\quadrangulation_\circ)$, F.~Chapoton~\cite{Chapoton-quadrangulations} raised three challenges: first prove that the dual graph of~$\accordionComplex(\quadrangulation_\circ)$, suitably oriented, has a lattice structure extending the Tamari and Cambrian lattices~\cite{Tamari,TamariFestschrift, Reading-CambrianLattices}; second construct geometric realizations of~$\accordionComplex(\quadrangulation_\circ)$ as fans and polytopes generalizing the known constructions of the associahedron; third show enumerative properties of the faces of~$\accordionComplex(\quadrangulation_\circ)$, among which he expects a bijection to exist between the facets of~$\accordionComplex(\quadrangulation_\circ)$ and other combinatorial objects called~\emph{serpent nests}.
These three challenges are evoked in the introduction of~\cite{Chapoton-quadrangulations} respectively at paragraph 22, last paragraph and paragraph 15. The serpent nest conjecture is also a specialization of~\cite[Conjecture 45]{Chapoton-quadrangulations} for~$x=y=1$. Serpent nests are essentially special sets of dual paths in the dual tree of the reference quadrangulation~$\quadrangulation_\circ$. As for the two other challenges, their study is related to extensions of known phenomena on the associahedron. Serpent nests are indeed expected by F.~Chapoton to play the same role towards Stokes complexes as nonnesting partitions towards associahedra. The serpent nest conjecture therefore morally asserts that the fact that nonnesting partitions are in bijection with triangulations of convex polygons holds in the more general context of Stokes complexes.

In~\cite{GarverMcConville}, A.~Garver and T.~McConville defined and studied the accordion complex~$\accordionComplex(\dissection_\circ)$ of any reference dissection~$\dissection_\circ$. Our presentation slightly differs from their's as they use a compatibility condition on the dual tree of the dissection~$\dissection_\circ$, but the simplicial complex is the same. In this context, they settled F.~Chapoton's lattice question, using lattice quotients of a lattice of biclosed sets. In a paper of T.~Manneville and V.~Pilaud~\cite{MannevillePilaud-geometricRealizationsAccordionComplexes}, geometric realizations (as fans and convex polytopes) of~$\accordionComplex(\dissection_\circ)$ were given for any reference dissection~$\dissection_\circ$, providing in particular an answer to F.~Chapoton's geometric question. The present paper settles the serpent nest conjecture of F.~Chapoton, in a version extended to any accordion complex. Other enumerative conjectures involving a statistic called~\emph{$F$-triangle}  are proposed in~\cite{Chapoton-quadrangulations}. A proof that this statistic is preserved by the~\emph{twist} operation~\cite[Conjecture 2.6]{Chapoton-quadrangulations} can be found in~\cite[Section 8.3.2]{Manneville-thesis}, but this result should go together with others that remain open for the moment.

\subsection{Overview}
\label{subsec:overview}

Section~\ref{sec:accordionComplex} introduces the accordion complex of a dissection~$\dissection_\circ$. We follow the presentation already adopted in~\cite{MannevillePilaud-geometricRealizationsAccordionComplexes}, where the definitions and arguments of A.~Garver and T.~McConville~\cite{GarverMcConville} are adapted to work directly on the dissection~$\dissection_\circ$ rather than on its dual graph. We define serpent nests in Section~\ref{sec:serpentNestConjecture} and present there our bijection between the facets of~$\accordionComplex(\dissection_\circ)$ and the serpent nests of~$\dissection_\circ$.
\newpage
\section{Accordion dissections}
\label{sec:accordionComplex}

\begin{figure}[b]
	\centerline{\includegraphics[width=1.15\textwidth]{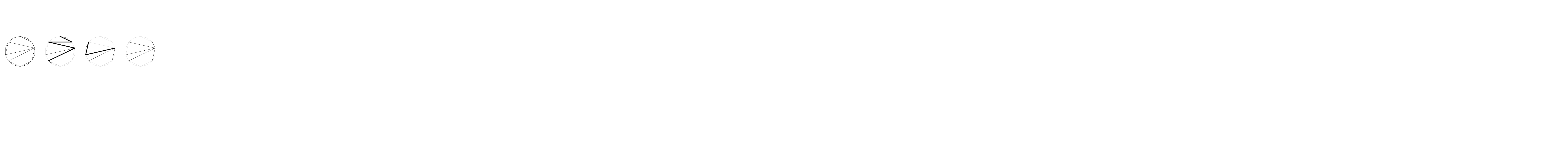}}
	\caption{A dissection (left) and three accordions with bold zigzags (middle and right).}
	\label{fig:exmAccordions}
\end{figure}

By a~\defn{diagonal} of a convex polygon~$\polygon$, we mean either an internal diagonal or an external diagonal (boundary edge) of~$\polygon$, but a~\defn{dissection}~$\dissection$ of~$\polygon$ is a set of pairwise noncrossing~\emph{internal} diagonals of~$\polygon$. We denote diagonals as pairs~$(i,j)$ of vertices, with~$i\le j$ when the labels on vertices are ordered. We moreover denote by~$\bar\dissection$ the dissection~$\dissection$ together with all boundary edges of~$\polygon$. The \defn{cells} of~$\dissection$ are the bounded connected components of the plane minus the diagonals of~$\bar\dissection$. An~\defn{accordion} of~$\dissection$ is a subset of~$\bar\dissection$ which contains either no or two incident diagonals in each cell of~$\dissection$. A~\defn{subaccordion} of~$\dissection$ is a subset of~$\dissection$ formed by the diagonals between two given internal diagonals in an accordion of~$\dissection$. A~\defn{zigzag} of~$\dissection$ is a subset~$\{\delta_0, \dots, \delta_{p+1}\}$ of~$\dissection$ where~$\delta_i$ shares distinct endpoints with~$\delta_{i-1}$ and~$\delta_{i+1}$ and separates them for any~$i \in [p]$. The~\defn{zigzag} of an accordion~$\accordion$ is the subset of the diagonals of~$\accordion$ which disconnect~$\accordion$. Notice that accordions of~$\dissection$ contain boundary edges of~$\polygon$, but not subaccordions nor zigzags. See \fref{fig:exmAccordions} for illustrations.

Consider $2n$ points on the unit circle labeled clockwise by~$1_\circ, 2_\bullet, 3_\circ, 4_\bullet, \dots, (2n-1)_\circ, (2n)_\bullet$ (with labels meant modulo~$2n$). We say that~$1_\circ, \dots, (2n-1)_\circ$ are the~\defn{hollow vertices} while~$2_\bullet, \dots, (2n)_\bullet$ are the~\defn{solid vertices}. The~\defn{hollow polygon} is the convex hull~$\polygon_\circ$ of~$1_\circ, \dots, {(2n-1)_\circ}$ while the~\defn{solid polygon} is the convex hull~$\polygon_\bullet$ of~$2_\bullet, \dots, (2n)_\bullet$. We simultaneously consider~\defn{hollow diagonals}~$\delta_\circ$ (with two hollow vertices) and~\defn{solid diagonals}~$\delta_\bullet$ (with two solid vertices), but never consider diagonals with vertices of each kind. Similarly, we consider~\defn{hollow dissections}~$\dissection_\circ$ (with only hollow diagonals) and~\defn{solid dissections}~$\dissection_\bullet$ (with only solid diagonals), but never mix hollow and solid diagonals in a dissection. To distinguish them more easily, hollow (resp. solid) vertices and diagonals appear red (resp. blue) in all pictures.

Let~$\dissection_\circ$ be an arbitrary reference hollow dissection. A~\defn{$\dissection_\circ$-accordion diagonal} is a solid diagonal~$\delta_\bullet$ such that the hollow diagonals of~$\bar\dissection_\circ$ crossed by~$\delta_\bullet$ form an accordion of~$\dissection_\circ$. In other words, $\delta_\bullet$ cannot enter and exit a cell of~$\dissection_\circ$ using two nonincident diagonals. It is also clearly equivalent to just require that~$\delta_\bullet$ crosses a connected subset of diagonals of~$\bar\dissection_\circ$. Note for instance that for any hollow diagonal~$(i_\circ, j_\circ) \in \bar\dissection_\circ$, the solid diagonals~$((i-1)_\bullet, (j-1)_\bullet)$ and~$((i+1)_\bullet, (j+1)_\bullet)$ are~$\dissection_\circ$-accordion diagonals. In particular, all boundary edges of~$\polygon_\bullet$ are $\dissection_\circ$-accordion diagonals. A~\defn{$\dissection_\circ$-accordion dissection} is a set of pairwise noncrossing internal $\dissection_\circ$-accordion diagonals, and we call~\defn{accordion complex} of~$\dissection_\circ$ the simplicial complex~$\accordionComplex(\dissection_\circ)$ of $\dissection_\circ$-accordion dissections.

\begin{example}
Consider the reference dissection~$\dissection_\circ^{\ex}$ of~\fref{fig:exmAccordionDissections}\,(left). Examples of (inclusion) maximal~$\dissection_\circ^{\ex}$-accordion dissections are given in \fref{fig:exmAccordionDissections}\,(middle right and right).

\begin{figure}
	\centerline{\includegraphics[width=1.2\textwidth]{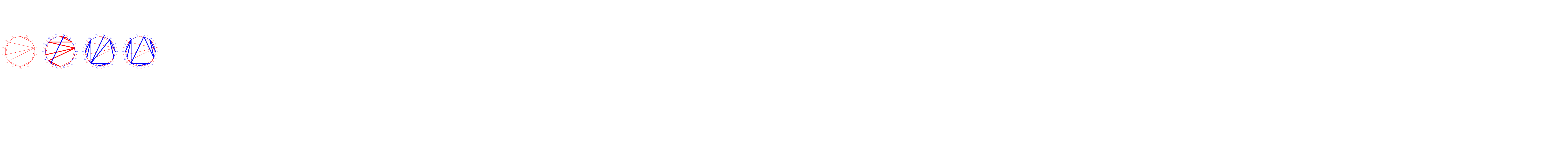}}
	\caption{A hollow dissection~$\dissection_\circ^{\ex}$, a solid~$\dissection_\circ^{\ex}$-accordion diagonal whose corresponding hollow accordion is bold, and two maximal solid~$\dissection_\circ^{\ex}$-accordion dissections.}
	\label{fig:exmAccordionDissections}
	\vspace{-.5cm}
\end{figure}

\end{example}

\begin{remark}
\label{rem:specialReferenceDissections}
Special reference hollow dissections~$\dissection_\circ$ give rise to special accordion complexes~$\accordionComplex(\dissection_\circ)$:
\begin{compactenum}[(i)]
\item If~$\dissection_\circ$ is the empty dissection with the whole hollow polygon as unique cell, then the $\dissection_\circ$-accordion complex~$\accordionComplex(\dissection_\circ)$ is reduced to the empty $\dissection_\circ$-accordion dissection.
\item If~$\dissection_\circ$ has a unique diagonal, then the accordion complex~$\accordionComplex(\dissection_\circ)$ is a segment.
\item For any hollow triangulation~$\triangulation_\circ$, all solid diagonals are $\triangulation_\circ$-accordions, so that the~$\triangulation_\circ$-accordion complex~$\accordionComplex(\triangulation_\circ)$ is the simplicial associahedron.
\item For any hollow quadrangulation~$\quadrangulation_\circ$, a solid diagonal is a $\quadrangulation_\circ$-accordion if and only if it never crosses two opposite edges of a quadrangle of~$\quadrangulation_\circ$, so that the accordion complex~$\accordionComplex(\quadrangulation_\circ)$ is the Stokes complex defined by Y.~Baryshnikov~\cite{Baryshnikov} and studied by F.~Chapoton~\cite{Chapoton-quadrangulations}.
\end{compactenum}
\end{remark}

We recall that the~\defn{dual tree}~$\dissection_\circ^\dual$ of the dissection~$\dissection_\circ$ is the tree whose nodes are the cells of~$\dissection_\circ$ and whose edges connects pairs of cells sharing a common diagonal of~$\dissection_\circ$. In particular the edges of~$\dissection_\circ^\dual$ are naturally identified to the diagonals of~$\dissection_\circ$.
Following the original definition of the noncrossing complex of A.~Garver and T.~McConville~\cite{GarverMcConville}, the accordion complex could equivalently be defined in terms of the dual tree~$\dissection_\circ^\dual$ of~$\dissection_\circ$.

\begin{remark}
\label{rem:reduction}
Assume that~$\dissection_\circ$ has a cell~$\cell_\circ$ containing~$p$ boundary edges of the hollow polygon~$\polygon_\circ$. Let~$\cell_\circ^1, \dots, \cell_\circ^p$ denote the $p$ (possibly empty) connected components of the hollow polygon minus~$\cell_\circ$. For~$i \in [p]$, let~$\dissection_\circ^i$ denote the dissection formed by the cell~$\cell_\circ$ together with the cells of~$\dissection_\circ$ in~$\cell_\circ^i$. Since no $\dissection_\circ$-accordion can contain internal diagonals crossing diagonals of distinct dissections~$\dissection_\circ^i$ and~$\dissection_\circ^j$ (with~$i \ne j$), the accordion complex of~$\dissection_\circ$ decomposes as the join:~${\accordionComplex(\dissection_\circ) = \accordionComplex(\dissection_\circ^1) * \cdots * \accordionComplex(\dissection_\circ^p)}$. In particular, we have the following reduction:
if a nontriangular cell of~$\dissection_\circ$ has two consecutive boundary edges~$\gamma_\circ, \delta_\circ$ of the hollow polygon, then contracting~$\gamma_\circ$ and~$\delta_\circ$ to a single boundary edge preserves the accordion complex of~$\dissection_\circ$.
\end{remark}

\begin{remark}
\label{rem:links}
The links in an accordion complex are joins of accordion complexes. Namely, consider a $\dissection_\circ$-accordion dissection~$\dissection_\bullet$ with cells~$\cell_\bullet^1, \dots, \cell_\bullet^p$. Let~$\dissection_\circ^i$ denote the hollow dissection obtained from~$\dissection_\circ$ by contracting all hollow (internal and external) diagonals which do not cross an edge of~$\cell_\bullet^i$. Then the link of~$\dissection_\bullet$ in~$\accordionComplex(\dissection_\circ)$ is clearly isomorphic to the join~$\accordionComplex(\dissection_\circ^1) * \dots * \accordionComplex(\dissection_\circ^p)$. \fref{fig:bijectionDissectionsSerpentNests} (left and middle left) illustrates how to visualize the link of two~$\dissection_\circ$-accordion diagonals.
\end{remark}

A bunch of combinatorial and geometric properties of accordion complexes were studied in~\cite{GarverMcConville,MannevillePilaud-geometricRealizationsAccordionComplexes,Manneville-thesis}. However we skip their presentation for sake of conciseness, as it appears that the proof of our result only relies on Remark~\ref{rem:links} (and on Remark~\ref{rem:reduction} for convenience).

\section{The serpent nest conjecture}
\label{sec:serpentNestConjecture}

For arbitrary vertices represented by residues modulo~$2n$, we mean by~$u<v<w$ that~$u,v$ and~$w$ are positioned in this order in clockwise cyclic order. In particular~$u<v<w$ is equivalent to~$w<u<v$ and~$v<w<u$. We also denote cyclic intervals by~$[u,w]\eqdef\{v\,|\,u\le v \le w\}$, and cyclic hollow (resp. solid) intervals by~$[u_\circ,w_\circ]_\circ\eqdef\{v_\circ\in\{1_\circ,\dots,(2n-1)_\circ\}\,|\,u_\circ\le v_\circ \le w_\circ\}$ (resp.~$[u_\bullet,w_\bullet]_\bullet\eqdef\{v_\bullet\in\{1_\bullet,\dots,(2n-1)_\bullet\}\,|\,u_\bullet\le v_\bullet \le w_\bullet\}$), where weak comparison symbols are extended accordingly to the previous notation. Finally we keep the notation~$u<v$ between residues modulo~$2n$ to denote the corresponding relation between their representatives in~$[2n]$.

\subsection{Serpents and serpent nests}
\label{subsec:serpentsAndSerpentNests}

We focus now on objects called~\emph{serpent nests} in~\cite{Chapoton-quadrangulations}. Recall that we denote by~$\dissection_\circ^\dual$ the dual tree of a reference hollow dissection~$\dissection_\circ$, whose vertices are the cells of~$\dissection_\circ$ and whose edges are the pairs of cells of~$\dissection_\circ$ that share a common diagonal of~$\dissection_\circ$. From now on we identify the edges of~$\dissection_\circ^\dual$ with the diagonals of~$\dissection_\circ$ in the natural way. A~\defn{serpent} of~$\dissection_\circ$ is an nonempty undirected dual path~$\serpent$ in~$\dissection_\circ^\dual$ whose edges (considered as hollow diagonals of~$\dissection_\circ$) form a subaccordion of~$\dissection_\circ$. Informally~$\serpent$ is a path in~$\dissection_\circ$ going through cells of~$\dissection_\circ$ by incident diagonals. The edges of~$\serpent$ not disconnecting it as a path (its ``end edges'') are its~\defn{final edges} (see Figure~\ref{fig:exmSerpentNest} left for an illustration).

\begin{figure}
\centerline{\includegraphics[height=6cm]{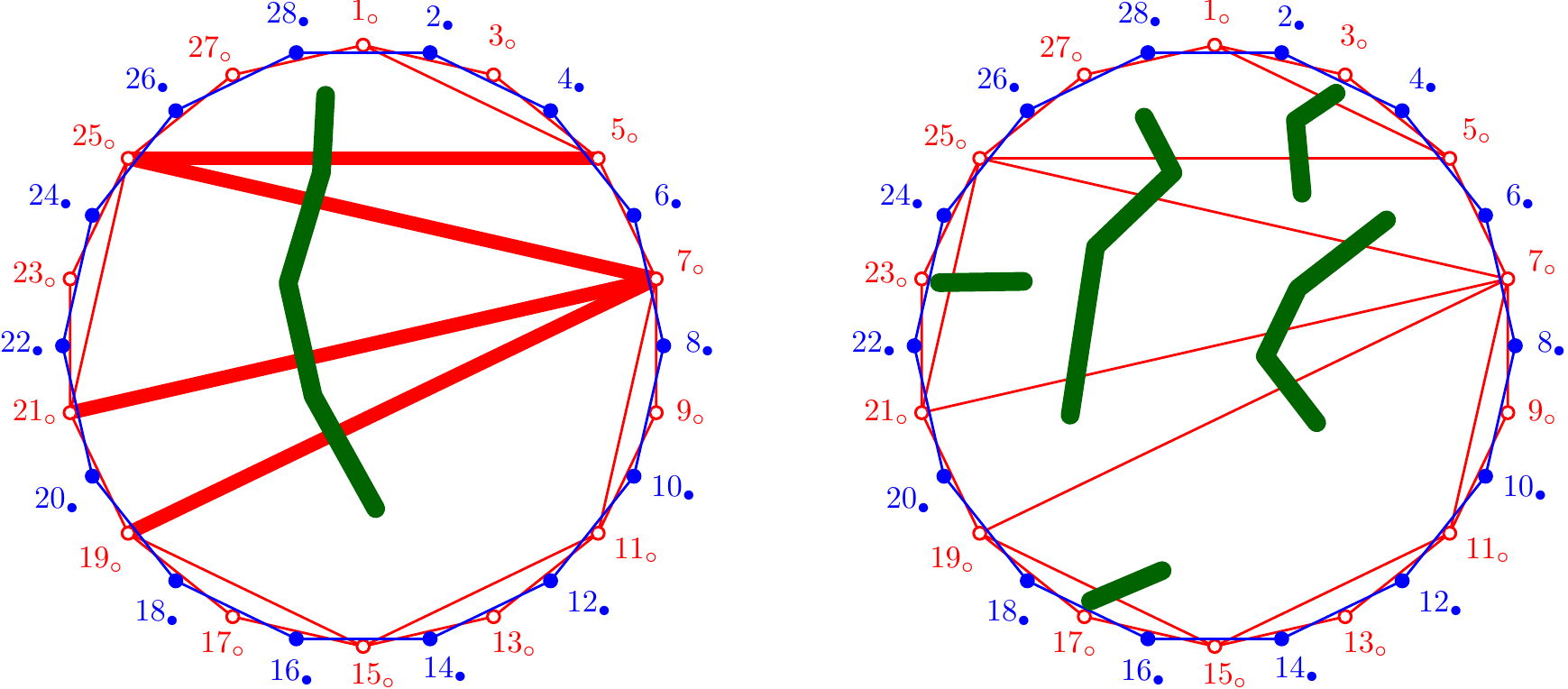}}
\caption{A (dark green) serpent~$\serpent$ with final edges~$(5_\circ,25_\circ)$ and~$(7_\circ,19_\circ)$ in the hollow dissection~$\dissection_\circ^{\ex}$ (left, the subaccordion crossed by~$\serpent$ is bold), and a serpent nest in~$\dissection_\circ^{\ex}$ (right).}
\label{fig:exmSerpentNest}
\vspace{-.4cm}
\end{figure}

Two serpents~$\serpent_1,\serpent_2$ are~\defn{incompatible} if they share at least one edge, so that the serpent~$\serpent_1\cap\serpent_2$ has for final edges diagonals~$(u_\circ^h,v_\circ^h)$ and~$(u_\circ^t,v_\circ^t)$ of~$\dissection_\circ$ with~$u_\circ^h< v_\circ^h\le u_\circ^t< v_\circ^t$, and if they satisfy either of the following conditions, where~$\serpent_1$ and~$\serpent_2$ (resp~$(u_\circ^h,v_\circ^h)$ and~$(u_\circ^t,v_\circ^t)$) may be exchanged.
\begin{enumerate}
 \item \label{cond:serpentNestHeads} The serpents~$\serpent_1$ and~$\serpent_2$ have a common final edge (Figure~\ref{fig:conditionsSerpentNest} left).
 \item \label{cond:serpentNestPatternLong1} The serpents~$\serpent_1$ and~$\serpent_2$ ``cross''. Formally~$\serpent_1$ simultaneously contains two diagonals incident to~$u_\circ^h$ and two diagonals incident to~$u_\circ^t$, and~$\serpent_2$ simultaneously contains two diagonals incident to~$v_\circ^h$ and two diagonals incident to~$v_\circ^t$ (Figure~\ref{fig:conditionsSerpentNest}~\mbox{middle}~left).
 \item \label{cond:serpentNestPatternLong2} The diagonal~$(u_\circ^h,v_\circ^h)$ is a final edge of the serpent~$\serpent_2$, and the serpent~$\serpent_1$ simultaneously contains two diagonals of~$\dissection_\circ$ incident to~$u_\circ^h$ (resp.~$v_\circ^h$) and two diagonals of~$\dissection_\circ$ incident to~$u_\circ^t$ (resp.~$v_\circ^t$) (Figure~\ref{fig:conditionsSerpentNest} middle right).
 \item \label{cond:serpentNestPatternShort} The diagonal~$(u_\circ^h,v_\circ^h)$ is a final edge of~$\serpent_1$, the diagonal~$(u_\circ^t,v_\circ^t)$ is a final edge of~$\serpent_2$, $\serpent_1$ contains two diagonals incident to~$u_\circ^t$ (resp.~$v_\circ^t$), and~$\serpent_2$ contains two diagonals incident to~$v_\circ^h$ (resp.~~$u_\circ^t$) (Figure~\ref{fig:conditionsSerpentNest} right).
\end{enumerate}

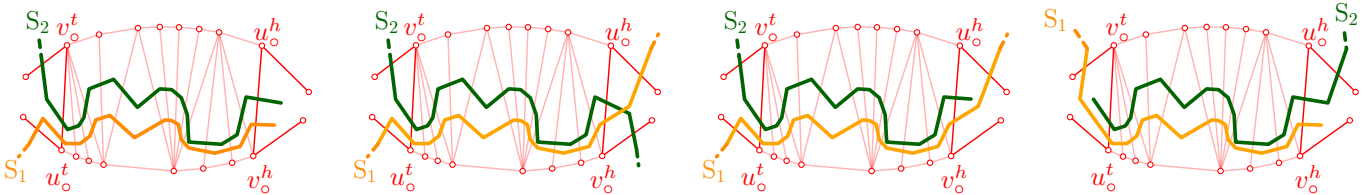
\begin{figure}[b]
\centerline{\begin{overpic}[width=1.2\textwidth]{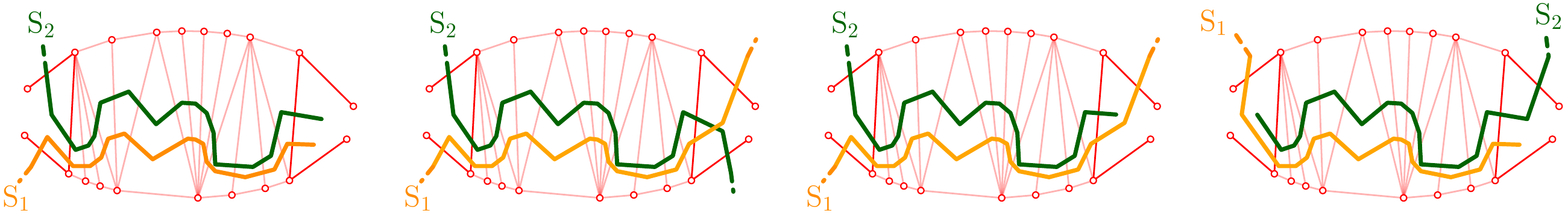}
\put(18.5,10.6){\textcolor{red}{$u_\circ^h$}}
\put(18,-.3){\textcolor{red}{$v_\circ^h$}}
\put(3.2,-0.1){\textcolor{red}{$u_\circ^t$}}
\put(4,11){\textcolor{red}{$v_\circ^t$}}
\put(44.3,10.6){\textcolor{red}{$u_\circ^h$}}
\put(43.2,-.3){\textcolor{red}{$v_\circ^h$}}
\put(28.5,-0.1){\textcolor{red}{$u_\circ^t$}}
\put(29.5,11){\textcolor{red}{$v_\circ^t$}}
\put(70.1,10.6){\textcolor{red}{$u_\circ^h$}}
\put(69,-.3){\textcolor{red}{$v_\circ^h$}}
\put(54.3,-0.1){\textcolor{red}{$u_\circ^t$}}
\put(55.3,11){\textcolor{red}{$v_\circ^t$}}
\put(95.5,10.6){\textcolor{red}{$u_\circ^h$}}
\put(95,-.3){\textcolor{red}{$v_\circ^h$}}
\put(80,-0.1){\textcolor{red}{$u_\circ^t$}}
\put(81.1,11){\textcolor{red}{$v_\circ^t$}}
\end{overpic}}
\vspace{.1cm}

\caption{Pairs of serpents~$\serpent_1$ (light yellow) and~$\serpent_2$ (dark green) incompatible because of Condition~\ref{cond:serpentNestHeads} (left), Condition~\ref{cond:serpentNestPatternLong1} (middle left), Condition~\ref{cond:serpentNestPatternLong2} (middle right) and Condition~\ref{cond:serpentNestPatternShort} (right).}
\label{fig:conditionsSerpentNest}
\vspace{-.4cm}
\end{figure}

The serpents~$\serpent_1$ and~$\serpent_2$ are~\defn{compatible} if they are not incompatible and a~\defn{serpent nest} of~$\dissection_\circ$ is a (potentially empty) set of pairwise compatible serpents (see Figures~\ref{fig:exmSerpentNest} left). Informally a set of serpents is a serpent nest if all its serpents can be simultaneously drawn as pairwise noncrossing dual paths in~$\dissection_\circ$ (Figure~\ref{fig:conditionsSerpentNest} middle left), with the additional conditions that no two of them ``end up in the same cell by entering it through a same diagonal of~$\dissection_\circ$'' (Figure~\ref{fig:conditionsSerpentNest} left) and that no serpent ``goes over the head of another serpent'' (Figure~\ref{fig:conditionsSerpentNest} middle right and right). To see that this description is indeed equivalent to the actual definition, observe that a serpent nest induces a unique valid\footnote{Informally a pattern ``avoiding'' Conditions~\ref{cond:serpentNestHeads},~\ref{cond:serpentNestPatternLong1} ,~\ref{cond:serpentNestPatternLong2} and~\ref{cond:serpentNestPatternShort}.} local pattern at each side of each internal diagonal of~$\dissection_\circ$, which immediately describes how to suitably draw all serpents. Figure~\ref{fig:serpentNestPatterns} (left) illustrates what such a local valid pattern typically looks like while Figure~\ref{fig:serpentNestPatterns} (middle and right) describes the forbidden local patterns.

\begin{figure}
\vspace{.4cm}

\centerline{\begin{overpic}[width=1.25\textwidth]{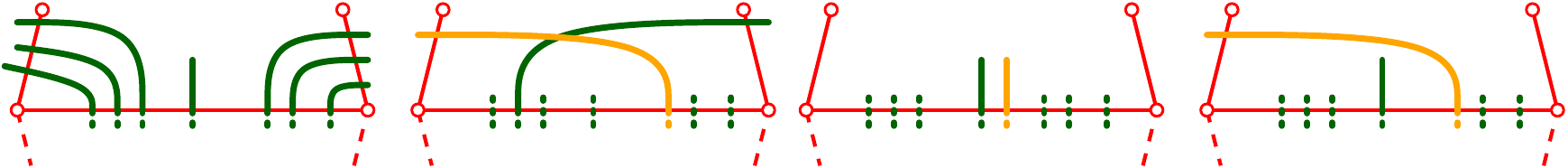}
\put(11.5,10){\textcolor{red}{$\cell_\circ$}}
\put(36,10){\textcolor{red}{$\cell_\circ$}}
\put(61,10){\textcolor{red}{$\cell_\circ$}}
\put(87,10){\textcolor{red}{$\cell_\circ$}}
\put(10,1.8){\textcolor{red}{$\delta_\circ$}}
\put(35.6,1.8){\textcolor{red}{$\delta_\circ$}}
\put(60,1.8){\textcolor{red}{$\delta_\circ$}}
\put(86,1.8){\textcolor{red}{$\delta_\circ$}}
\end{overpic}}
\caption{A valid serpent nest pattern at an internal diagonal~$\delta_\circ$ of a cell~$\cell_\circ$ in a hollow dissection (left) and the three obstructions to valid patterns (middle and right).}
\label{fig:serpentNestPatterns}
\vspace{-.5cm}
\end{figure}
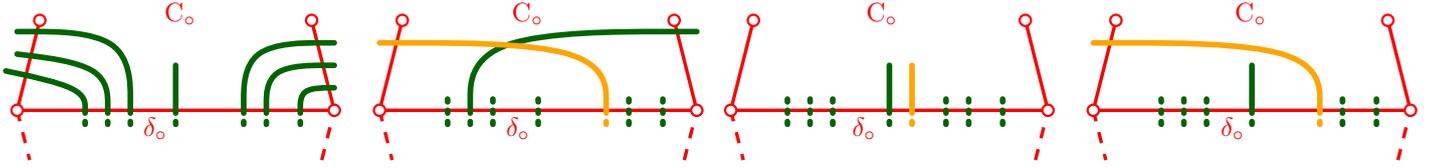

\subsection{Main result}

The serpent nest conjecture was initially stated by F.~Chapoton only for reference hollow quadrangulations. Let us first rephrase it in our setting.

\begin{conjecture}{{\cite[Conjecture 45 for $x=y=1$]{Chapoton-quadrangulations}}}
For any hollow quadrangulation~$\quadrangulation_\circ$, there is a bijection between the serpent nests of~$\quadrangulation_\circ$ and the maximal~$\quadrangulation_\circ$-accordion dissections.
\end{conjecture}

We prove this conjecture for arbitrary reference hollow dissections.

\begin{theorem}
\label{thm:bijectionDissectionsSerpentNests}
For any hollow dissection~$\dissection_\circ$, there is a bijection between the serpent nests of~$\dissection_\circ$ and the maximal~$\dissection_\circ$-accordion dissections.
\end{theorem}

We need the following observation to prove Theorem~\ref{thm:bijectionDissectionsSerpentNests}.

\begin{lemma}
\label{lem:inductionTriangle}
Let~$\dissection_\circ$ be a hollow dissection, let~$\cell_\circ$ be a triangular cell of~$\dissection_\circ$ which is a leaf of~$\dissection_\circ^\dual$, whose unique internal diagonal of~$\dissection_\circ$ is~$(1_\circ,5_\circ)$. For any maximal~$\dissection_\circ$-accordion dissection~$\dissection_\bullet$, there exists a unique solid vertex~$x_\bullet>4_\bullet$ such that both solid (internal or external) diagonals~$(2_\bullet,x_\bullet)$ and~$(4_\bullet,x_\bullet)$ belong to~$\bar\dissection_\bullet$.
\end{lemma}

\begin{proof}
For any two solid vertices~$4_\bullet<x_\bullet<x_\bullet'$, the solid diagonals~$(2_\bullet,x_\bullet)$ and~$(4_\bullet,x_\bullet')$ cross, which settles the uniqueness part. Let~$x_\bullet$ be the smallest element in~$[6_\bullet,(2n)_\bullet]_\bullet$ such that~$(2_\bullet,x_\bullet)$ is a (internal or external) diagonal of~$\bar\dissection_\bullet$. Observe that the solid diagonal~$(4_\bullet,x_\bullet)$ crosses no diagonal of~$\dissection_\bullet$, since such a diagonal should be of the form~$(2_\bullet,y_\bullet)$ with~$y_\bullet<x_\bullet$, contradicting the minimality of~$x_\bullet$. Moreover, the diagonal~$(4_\bullet,x_\bullet)$ crosses the same set of hollow diagonals as~$(2_\bullet,x_\bullet)$ but with the external diagonal~$(1_\circ,3_\circ)$ replaced by~$(3_\circ,5_\circ)$, so that it is a~$\dissection_\circ$-accordion solid diagonal. This concludes the proof since then~$(4_\bullet,x_\bullet)\in\dissection_\bullet$ by maximality of~$\dissection_\bullet$.
\end{proof}

The assumptions in Lemma~\ref{lem:inductionTriangle} do not introduce any real restriction on the number of vertices of~$\cell_\circ$ nor on its unique internal diagonal. Indeed Remark~\ref{rem:reduction} allows us to assume that~$\cell_\circ$ is triangular as soon as it is a dual leaf, and we may rotate the labels of the vertices of the polygon~$\polygon$ in order that the internal diagonal of~$\cell_\circ$ is~$(1_\circ,5_\circ)$. We thus keep these assumptions in the proof of Theorem~\ref{thm:bijectionDissectionsSerpentNests}. It consists in an induction relying on Lemma~\ref{lem:inductionTriangle} and the description of links in accordion complexes given in Remark~\ref{rem:links}. Informally we decompose any maximal~$\dissection_\circ$-accordion dissection~$\dissection_\bullet$ into two parts, according to its distinguished vertex~$x_\bullet$ given by Lemma~\ref{lem:inductionTriangle} (Figure~\ref{fig:bijectionDissectionsSerpentNests} left and middle left), and find a corresponding serpent nest inductively in each of them (Figure~\ref{fig:bijectionDissectionsSerpentNests} middle right). One must then remark that the two serpent nests in these two parts can be ``unfolded'' and gathered into a valid serpent nest~$\serpentNest$ of~$\dissection_\circ$ (nonbold serpents in Figure~\ref{fig:bijectionDissectionsSerpentNests} right). We then add a last serpent to~$\serpentNest$, whose final edges are~$(1_\circ,5_\circ)$ and the farthest possible diagonal of the ``zigzag crossed by both~$(2_\bullet,x_\bullet)$ and~$(4_\bullet,x_\bullet)$'' such that the new serpent  (bold in Figure~\ref{fig:bijectionDissectionsSerpentNests} right) does not create a validity obstruction in the local patterns inherited from~$\serpentNest$.

\begin{figure}
\centerline{\begin{overpic}[width=1.2\textwidth]{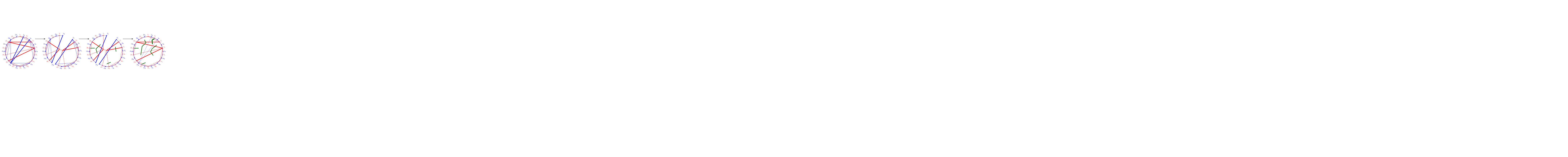}\end{overpic}}
\caption{A maximal~$\dissection_\circ^{\ex}$-accordion dissection~$\dissection_\bullet$ (left), and the serpent nest~$\bijectionDissectionsToSerpentNests{\dissection_\circ^{\ex}}(\dissection_\bullet)$ of~$\dissection_\circ^{\ex}$ defined in the proof of Theorem~\ref{thm:bijectionDissectionsSerpentNests} (right). Here the solid vertex~$x_\bullet$ given by Lemma~\ref{lem:inductionTriangle} is~$18_\bullet$ and the zigzag~$\zigzag_\circ$ defined in the proof of Theorem~\ref{thm:bijectionDissectionsSerpentNests} contains~$3$ (bold) diagonals. The bijection~$\bijectionDissectionsToSerpentNests{.}$ is applied inductively in each part of the link of~$\{(2_\bullet,18_\bullet),(4_\bullet,18_\bullet)\}$ in~$\accordionComplex(\dissection_\circ^{\ex})$ to obtain serpent nests whose serpents do not cross the diagonals~$(2_\bullet,18_\bullet)$ and~$(4_\bullet,18_\bullet)$ (middle). All these serpents are ``unfolded'' into a valid serpent nest in~$\dissection_\circ^{\ex}$ (right), to which a (bold) serpent ``between the diagonals~$(2_\bullet,18_\bullet)$ and~$(4_\bullet,18_\bullet)$'' is added. The final edges of this additional serpent are~$(1_\circ,5_\circ)$ and~$(5_\circ,25_\circ)$, which is the farthest diagonal of~$\zigzag_\circ$ after which it can end in order to be compatible with the serpents inductively obtained.}
\label{fig:bijectionDissectionsSerpentNests}
\vspace{-.8cm}
\end{figure}

\begin{proof}[Proof of Theorem~\ref{thm:bijectionDissectionsSerpentNests}]
The proof is by induction on the number of diagonals in~$\dissection_\circ$. If~$\dissection_\circ$ is empty, then the unique~$\dissection_\circ$-accordion dissection is the empty set and the unique serpent nest in~$\dissection_\circ$ is also the empty set. Now as previously explained, we can assume for the induction step that~$\dissection_\circ$ contains the diagonal~$(1_\circ,5_\circ)$ (since its dual tree has a leaf). Let~$\accordionComplex^{\max}(\dissection_\circ)$ be the set of maximal~$\dissection_\circ$-accordion dissections and~$\allSerpentNests(\dissection_\circ)$ be the set of serpent nests of~$\dissection_\circ$. We define bijections~$\bijectionDissectionsToSerpentNests{\dissection_\circ}:\accordionComplex^{\max}(\dissection_\circ)\rightarrow\allSerpentNests(\dissection_\circ)$ and~$\bijectionSerpentNestsToDissections{\dissection_\circ}:\allSerpentNests(\dissection_\circ)\rightarrow\accordionComplex^{\max}(\dissection_\circ)$ that are reverse to each other as follows.

Let~$\dissection_\bullet\in\accordionComplex^{\max}(\dissection_\circ)$ and let~$x_\bullet$ be the solid vertex such that~$\{(2_\bullet,x_\bullet),(4_\bullet,x_\bullet)\}\subseteq\dissection_\bullet$, given by Lemma~\ref{lem:inductionTriangle}. Let~$\dissection_\circ^{>x_\bullet}$ (resp~$\dissection_\circ^{<x_\bullet}$) be the hollow dissection obtained by contracting all diagonals of~$\dissection_\circ$ with both endpoints in~$[3_\circ,(x-1)_\circ]_\circ$ (resp.~$[(x+1)_\circ,3_\circ]_\circ$) into a single vertex~$c_\circ^1$ (resp.~$c_\circ^2$). Let~$\dissection_\bullet^{>x_\bullet}\in\accordionComplex^{\max}(\dissection_\circ^{>x_\bullet})$ (resp.~$\dissection_\bullet^{<x_\bullet}\in\accordionComplex^{\max}(\dissection_\circ^{<x_\bullet})$) be the dissection obtained by keeping only the diagonals of~$\dissection_\bullet$ with both endpoints in~$[x_\bullet,2_\bullet]_\bullet$ (resp.~$[4_\bullet,x_\bullet]_\bullet$). As both dissections~$\dissection_\circ^{>x_\bullet}$ and~$\dissection_\circ^{<x_\bullet}$ have less diagonals than~$\dissection_\circ$, the induction hypothesis provides us with bijections~$\bijectionDissectionsToSerpentNests{\dissection_\circ^{>x_\bullet}}:\accordionComplex^{\max}(\dissection_\circ^{>x_\bullet})\rightarrow\allSerpentNests(\dissection_\circ^{>x_\bullet})$ and~$\bijectionDissectionsToSerpentNests{\dissection_\circ^{<x_\bullet}}:\accordionComplex^{\max}(\dissection_\circ^{<x_\bullet})\rightarrow\allSerpentNests(\dissection_\circ^{<x_\bullet})$, whose reverse functions we respectively denote by~$\bijectionSerpentNestsToDissections{\dissection_\circ^{>x_\bullet}}:\allSerpentNests(\dissection_\circ^{>x_\bullet})\rightarrow\accordionComplex^{\max}(\dissection_\circ^{>x_\bullet})$ and~$\bijectionSerpentNestsToDissections{\dissection_\circ^{<x_\bullet}}:\allSerpentNests(\dissection_\circ^{<x_\bullet})\rightarrow\accordionComplex^{\max}(\dissection_\circ^{<x_\bullet})$. We then define
\[
\serpentNest_1\eqdef\bijectionDissectionsToSerpentNests{\dissection_\circ^{>x_\bullet}}(\dissection_\bullet^{>x_\bullet})\qquad\text{and}\qquad\serpentNest_2\eqdef\bijectionDissectionsToSerpentNests{\dissection_\circ^{<x_\bullet}}(\dissection_\bullet^{<x_\bullet}).
\]
Observe that a cell of~$\dissection_\circ$ different from~$(1_\circ,3_\circ,5_\circ)$ either contains at least two vertices in~$[3_\circ,(x-1)_\circ]_\circ$ and at most one in~$[(x+1)_\circ,3_\circ]_\circ$ or conversely, as both diagonals~$(2_\bullet,x_\bullet)$ and~$(4_\bullet,x_\bullet)$ cross accordions of~$\dissection_\circ$. The cells of~$\dissection_\circ$ are thus naturally partitioned and identified into the cells of~$\dissection_\circ^{>x_\bullet}$ and~$\dissection_\circ^{<x_\bullet}$. Moreover a subaccordion of~$\dissection_\circ^{>x_\bullet}$ (resp.~$\dissection_\circ^{<x_\bullet}$) naturally extends to a subaccordion of~$\dissection_\circ$ by replacing its diagonals~$(a_\circ^1,c_\circ^1)$ (resp.~$(a_\circ^2,c_\circ^2)$) by the set of all diagonals of the accordion crossed by~$(2_\bullet,x_\bullet)$ with~$a_\circ^1$ (resp.~$a_\circ^2$) as an endpoint. So serpents in~$\dissection_\circ^{>x_\bullet}$ (resp.~$\dissection_\circ^{<x_\bullet}$) are also naturally identified to some serpents in~$\dissection_\circ$. It is moreover clear that compatible serpents in~$\dissection_\circ^{>x_\bullet}$ (resp.~$\dissection_\circ^{<x_\bullet}$) extend to compatible serpents in~$\dissection_\circ$, and that any serpent in~$\dissection_\circ^{>x_\bullet}$ extends to a serpent in~$\dissection_\circ$ that is compatible with any serpent obtained by extending a serpent of~$\dissection_\circ^{<x_\bullet}$. We therefore abuse notations and still denote by~$\serpentNest_1\sqcup\serpentNest_2$ the corresponding serpent nests in~$\dissection_\circ$.
%
%
We first settle two degenerate cases.
\begin{itemize}
 \item If~$x_\bullet=(2n)_\bullet$, then we define~$\bijectionDissectionsToSerpentNests{\dissection_\circ}(\dissection_\bullet)=\serpentNest_1\sqcup\serpentNest_2$.
 \item If~$x_\bullet=6_\bullet$, then we define~$\bijectionDissectionsToSerpentNests{\dissection_\circ}(\dissection_\bullet)=\serpentNest_1\sqcup\serpentNest_2\sqcup\{\serpent\}$ where~$\serpent$ is the serpent of~$\dissection_\circ$ whose single edge corresponds to~$(1_\circ,5_\circ)$. It is clear that this serpent is compatible with all those in~$\serpentNest_1\sqcup\serpentNest_2$ since it does not share any common edge with them.
\end{itemize}
We are left with the case where both solid diagonals~$(2_\bullet,x_\bullet)$ and ~$(4_\bullet,x_\bullet)$ are internal. Let~$\zigzag_\circ=\{\delta_\circ^1,\dots\delta_\circ^\ell\}$ denote the zigzag of the accordion crossed by~$(2_\bullet,x_\bullet)$, where the diagonal~$(1_\circ,5_\circ)$ is considered as a boundary edge (and therefore not in~$\zigzag_\circ$), and such that~$\delta_\circ^i$ is incident to~$\delta_\circ^{i-1}$ and~$\delta_\circ^{i+1}$ for~$i\in[2,\ell-1]$. As we already dealt with the cases where~$x_\bullet\in\{6_\bullet,(2n)_\bullet\}$, the zigzag~$\zigzag_\circ$ is not empty, and we can assume by symmetry that~$5_\circ$ is an endpoint of~$\delta_\circ^1$. Let~$\serpent$ be the serpent of~$\dissection_\circ$ compatible with all serpents in~$\serpentNest_1\sqcup\serpentNest_2$ whose final edges are~$(1_\circ,5_\circ)$ and the diagonal~$\delta_\circ^{i_{\max}}$, where~$i_{\max}$ is maximal in~$[\ell]$ for this property. It is well-defined since
\begin{itemize}
 \item all dual paths in~$\dissection_\circ$ with final edges~$(1_\circ,5_\circ)$ and~$\delta_\circ^i$ ($i\in[\ell]$) are serpents of~$\dissection_\circ$, and
 \item the serpent with final edges~$(1_\circ,5_\circ)$ and~$\delta_\circ^1$ is compatible with all serpents in~$\serpentNest_1\sqcup\serpentNest_2$, by a quick case analysis.
\end{itemize}
We finally define
\[
\bijectionDissectionsToSerpentNests{\dissection_\circ}(\dissection_\bullet)\eqdef\serpentNest_1\sqcup\serpentNest_2\sqcup\{\serpent\}.
\]

To show that~$\bijectionDissectionsToSerpentNests{\dissection_\circ}$ is a bijection, we define its reverse bijection~$\bijectionSerpentNestsToDissections{\dissection_\circ}$. For this, we only need to show how to determine, given a serpent nest~$\serpentNest$ of~$\dissection_\circ$, the distinguished vertex~$x_\bullet$ of the maximal~$\dissection_\circ$-accordion dissection~$\bijectionSerpentNestsToDissections{\dissection_\circ}(\serpentNest)$. This~$x_\bullet$ should be chosen such that the serpents in~$\serpentNest$ then separate on each sides of the diagonals~$(2_\bullet,x_\bullet)$ and ~$(4_\bullet,x_\bullet)$, in order for us to conclude, using the reverse bijections~$\bijectionSerpentNestsToDissections{\dissection_\circ^{>x_\bullet}}:\allSerpentNests(\dissection_\circ^{>x_\bullet})\rightarrow\accordionComplex^{\max}(\dissection_\circ^{>x_\bullet})$ and~$\bijectionSerpentNestsToDissections{\dissection_\circ^{<x_\bullet}}:\allSerpentNests(\dissection_\circ^{<x_\bullet})\rightarrow\accordionComplex^{\max}(\dissection_\circ^{<x_\bullet})$. The way we determine the vertex~$x_\bullet$ is illustrated in Figure~\ref{fig:reverseBijection}.

\begin{figure}
\centerline{\includegraphics[width=1.2\textwidth]{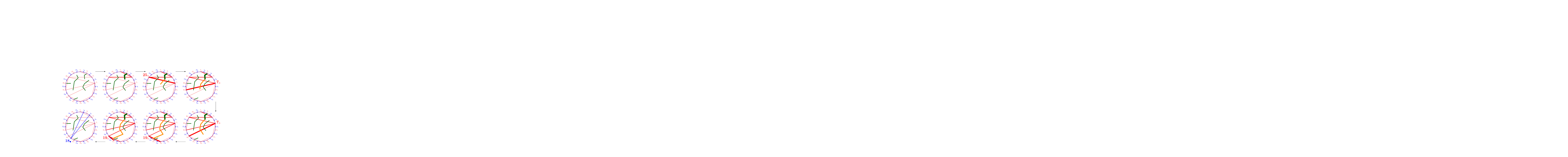}}
\caption{Illustration of the algorithm used to determine the image~$\bijectionSerpentNestsToDissections{\dissection_\circ^{\ex}}(\serpentNest)$ of a serpent nest~$\serpentNest$ of the hollow dissection~$\dissection_\circ^{\ex}$ (top left). In the successive pictures, the diagonals~$\gamma_i$'s of the proof appear bold, the currently defined one being double bold. At the end of the process, we are left with a vertex~$x_\bullet$ ($18_\bullet$ here) which allows to suitably separate the serpents of~$\serpentNest$ at each sides of the diagonals~$(2_\bullet,18_\bullet)$ and~$(4_\bullet,18_\bullet)$ (bottom left).}
\label{fig:reverseBijection}
\vspace{-1cm}
\end{figure}

The two ``degenerate'' cases where~$\serpentNest$ either contains no serpent containing~$(1_\circ,5_\circ)$ or the serpent whose unique edge is~$(1_\circ,5_\circ)$ are easily settled, as we dealt with them separately when defining~$\bijectionDissectionsToSerpentNests{\dissection_\circ}$. Suppose that~$\serpentNest$ contains a serpent~$\serpent$ with final edges~$(1_\circ,5_\circ)$ and a hollow diagonal~$\delta_\circ$. As~$(1_\circ,5_\circ)$ is incident to a dual leaf, there is no other serpent than~$\serpent$ in~$\serpentNest$ that contains it, since otherwise it would fulfill Condition~\ref{cond:serpentNestHeads} together with~$\serpent$. We now inductively define a sequence of hollow diagonals~$(\gamma_\circ^i)_{i\ge 1}$, such that for~$i\ge1$, the dual path from~$(1_\circ,5_\circ)$ to~$\gamma_\circ^i$ is a serpent, that we denote by~$\serpent_i$. In what follows, we denote by~$u_\circ^i$ the endpoint of~$\gamma_\circ^i$ contained in another edge of~$\serpent_i$.
\begin{itemize}
 \item Let~$\cell_\circ^1$ be the cell which is the endpoint (as dual node in~$\dissection_\circ^\dual$) of~$\serpent$ (as dual path in~$\dissection_\circ^\dual$) incident to~$\delta_\circ$. We let~$u_\circ^1$ be the endpoint of~$\delta_\circ$ not contained in another edge of~$\serpent$ and~$\gamma_\circ^1$ be the diagonal of~$\cell_\circ^1$ incident to~$\delta_\circ$ at~$u_\circ^1$. The edges of~$\serpent_1$ are then the edges of~$\serpent$ together with~$\gamma_\circ^1$, so that~$\gamma_\circ^1$ and~$u_\circ^1$ satisfy the required property.
 \item For~$i>1$, we consider the cell~$\cell_\circ^i$ which is the endpoint (as dual node in~$\dissection_\circ^\dual$) of~$\serpent_{i-1}$ (as dual path in~$\dissection_\circ^\dual$) incident to~$\gamma_\circ^{i-1}$. Let~$\lambda_\circ^i$ be the other diagonal of~$\cell_\circ^i$ containing~$u_\circ^{i-1}$. The dual path with final edges~$(1_\circ,5_\circ)$ and~$\lambda_\circ^i$ is then a serpent, that we denote by~$\serpent_{i-1}^+$. We distinguish two cases.
 	\begin{compactenum}[(i)]
 	 \item If~$\serpent_{i-1}^+$ is compatible with all serpents of~$\serpentNest\ssm\{\serpent\}$ not containing~$\lambda_\circ^i$, then we define~$\gamma_\circ^i\eqdef\lambda_\circ^i$ and~$u_\circ^i \eqdef u_\circ^{i-1}$.
 	 \item If a serpent in~$\serpentNest\ssm\{\serpent\}$ not containing~$\lambda_\circ^i$ is incompatible with~$\serpent_{i-1}^+$, then we let~$\gamma_\circ^i$ be the diagonal of~$\cell_\circ^i$ incident to~$\gamma_\circ^{i-1}$ different from~$\lambda_\circ^i$, which fulfills the required condition. Observe that in this case we necessarily have~$u_\circ^i \neq u_\circ^{i-1}$.
 	\end{compactenum}
\end{itemize}

Observe that any serpents~$\serpent_i$ ($i\ge1$) is compatible with all serpents of~$\serpentNest\ssm\{\serpent\}$ not containing~$\lambda_\circ^i$. This is clear for serpents~$\serpent_i$ obtained from Case~(i), and it follows from straightforward case analyses for~$\serpent_1$ and for serpents~$\serpent_i$ obtained from Case~(ii). The sequence~$(\gamma_\circ^i)_{i\ge1}$ cannot be infinite for there are finitely many hollow diagonals in~$\dissection_\circ$. It thus stops when the new diagonal~$\gamma_\circ^j$ (for some~$j\ge1$) that we want to define is an external hollow diagonal~$((x-1)_\circ,(x+1)_\circ)$ for some~$x_\bullet\in[2_\bullet,(2n)_\bullet]_\bullet$. In fact it is immediate that~$x_\bullet\in[8_\bullet,(2n-2)_\bullet]_\bullet$. Notice that the solid diagonal~$(2_\bullet,x_\bullet)$ crosses an accordion whose diagonals are those of~$\serpent_{j-1}$ (with~$\serpent_0=\serpent$ by convention) together with the external hollow diagonals~$(1_\circ,3_\circ)$ and~$((x-1)_\circ,(x+1)_\circ)$. Moreover the vertices of the zigzag of~$(2_\bullet,x_\bullet)$ are some vertices of edges of the serpent~$\serpent$ together with the vertices in~$\{u_\circ^i\,|\,i\ge0\}$, that we denote~$\{u_\circ^{i_1},\dots,u_\circ^{i_p}\}$ without repetition. Observe finally that the compatibility conditions on~$\serpent_j$ imply that any serpent of~$\serpentNest\ssm\{\serpent\}$ can be obtained by extending either a serpent of~$\dissection_\circ^{>x_\bullet}$ or of~$\dissection_\circ^{<x_\bullet}$ to~$\dissection_\circ$. Therefore~$\serpentNest\ssm\{\serpent\}$ splits into two serpent nests~$\serpentNest_1$ and~$\serpentNest_2$ in the hollow dissections~$\dissection_\circ^{>x_\bullet}$ and~$\dissection_\circ^{<x_\bullet}$ obtained from~$x_\bullet$. Let
\[
\bijectionSerpentNestsToDissections{\dissection_\circ}(\serpentNest)\eqdef\bijectionSerpentNestsToDissections{\dissection_\circ^{>x_\bullet}}(\serpentNest_1)\sqcup\bijectionSerpentNestsToDissections{\dissection_\circ^{<x_\bullet}}(\serpentNest_2)\sqcup\{(2_\bullet,x_\bullet),(4_\bullet,x_\bullet)\}.
\]
It remains to check that~$\bijectionSerpentNestsToDissections{\dissection_\circ}\circ\bijectionDissectionsToSerpentNests{\dissection_\circ}$ is the identity function on~$\accordionComplex^{\max}(\dissection_\circ)$. It is clear, from the definition in Cases~(i) and~(ii), that if~$u_\circ^i\neq u_\circ^{i-1}$ for some~$2\le i\le j$, then the serpent~$\serpent_{i-1}$ is incompatible with at least one serpent in~$\serpentNest\ssm\{\serpent\}$. Thus~$\serpent$ is the only serpent, among all serpents~$\serpent_i$ for $i\in\{0,i_1,\dots,i_p\}$, that is compatible with all serpents in~$\serpentNest\ssm\{\serpent\}$ and whose final edge different from~$(1_\circ,5_\circ)$ belongs to the zigzag of~$(2_\bullet,x_\bullet)$ (where~$(1_\circ,5_\circ)$ is considered as a boundary edge). This concludes the proof since it implies that the vertex~$x_\bullet$ given by Lemma~\ref{lem:inductionTriangle} is the same for~$\dissection_\bullet$ and~$\bijectionSerpentNestsToDissections{\dissection_\circ}\circ\bijectionDissectionsToSerpentNests{\dissection_\circ}(\dissection_\bullet)$.
\end{proof}

\section*{Acknowledgments}

I am grateful to Fr{\'e}d{\'e}ric Chapoton for interesting discussions about this problem and its motivations, that I did not develop here for sake of conciseness. This work was supported by a French doctoral grant Gaspard Monge of the {\'E}cole polytechnique (Palaiseau, France) and partially supported by the French ANR grant SC3A (15 CE40 0004 01).

\bibliographystyle{alpha}
\bibliography{accordionComplex}
\label{sec:biblio}

\end{document}